\documentclass[11pt]{amsart}
\usepackage{amsmath,amssymb}
\newtheorem{theorem}{Theorem}

\newtheorem{corollary}{Corollary}
\newtheorem{lemma}{Lemma}
\theoremstyle{definition}
\newtheorem{definition}{Definition}
\newtheorem{remark}{Remark}

\usepackage{latexsym,amssymb}
\begin{document}
\baselineskip=14.5pt
\title[On a generalization of Menon-Sury identity to number fields]{On a generalization of Menon-Sury identity to number fields involving a Dirichlet Character} 

\author{Jaitra Chattopadhyay and Subha Sarkar}
\address[Jaitra Chattopadhyay]{Indian Institute of Technology, Guwahati, Guwahati-781 039, Assam, India}
\email[Jaitra Chattopadhyay]{jaitra@iitg.ac.in, chat.jaitra@gmail.com}
\address[Subha Sarkar]{Ramakrishna Mission Vivekananda Educational and Research Institute, Belur Math, Howrah-711 202, West Bengal, India}
\email[Subha Sarkar]{subhasarkar2051993@gmail.com}

\begin{abstract}
For every positive integer $n$, Sita Ramaiah's identity states that
\medskip
\begin{equation*}
\sum_{a_1, a_2, a_1+a_2 \in (\mathbb{Z}/n\mathbb{Z})^*}  \gcd(a_1+a_2-1,n) = \phi_2(n)\sigma_0(n) \; \text{ where } \; \phi_2(n)= \sum_{a_1, a_2, a_1+a_2 \in (\mathbb{Z}/n\mathbb{Z})^*} 1,
\end{equation*}
\medskip
where $(\mathbb{Z}/n\mathbb{Z})^*$ is the multiplicative group of units of the ring $\mathbb{Z}/n\mathbb{Z}$ and $\sigma_s(n) = \displaystyle\sum_{d\mid n}d^s$.

\smallskip

This identity can also be viewed as a generalization of Menon's identity. In this article, we generalize this identity to an algebraic number field $K$ involving a Dirichlet character $\chi$. Our result is a further generalization of a recent result in \cite{wj} and \cite{sury}.
 
\end{abstract}

\renewcommand{\thefootnote}{}

\footnote{2020 \emph{Mathematics Subject Classification}: Primary 11R04, 11A25.}

\footnote{\emph{Key words and phrases}: Ring of algebraic integers, Dirichlet character}

\renewcommand{\thefootnote}{\arabic{footnote}}
\setcounter{footnote}{0}

\maketitle

\section{Introduction}
For a positive integer $n$, the classical Menon's identity \cite{menon} states that
\begin{equation}\label{menon}
\displaystyle\sum_{\substack{a=1 \\ \gcd(a,n)=1}}^n \gcd(a-1,n) =\phi(n)\sigma_0(n),
\end{equation}
where $\phi(n)$ is the Euler's totient function and $\sigma_0(n) =\displaystyle \sum_{d \mid n}1$. This identity has been generalized in many direction by several authors (cf. \cite{h,lk,m,toth2,toth3}).

\smallskip

In 2009, Sury \cite{sury} generalized \eqref{menon} and proved that
\begin{equation}\label{sury}
\sum_{\substack{1 \leq a, b_1, b_2, \ldots, b_s \leq n \\ \gcd(a,n)=1}} \gcd(a-1,b_1,\ldots, b_s,n)= \phi(n)\sigma_s(n), \; \text{ where }\; \sigma_s(n) =\displaystyle \sum_{d \mid n} d^s.
\end{equation}

\smallskip

For an integer $n \geq 2$ and a Dirichlet character $\chi$ modulo $n$ with conductor $d$, Cao and Zhao \cite{zc} recently derived the following identity.
\begin{equation}\label{zc}
\sum_{\substack{a=1 \\ \gcd(a,n)=1}}^n \gcd(a-1,n)\chi(a) = \phi(n) \sigma_0\left(\frac{n}{d}\right).
\end{equation}

Later, Li, Hu and Kim \cite{lhk} generalized \eqref{sury} and  \eqref{zc} and proved that
\begin{equation}\label{lhk}
\sum_{\substack{1 \leq a, b_1, b_2, \ldots, b_s \leq n \\ \gcd(a,n)=1}} \gcd(a-1,b_1,\ldots, b_s,n)\chi(a)= \phi(n)\sigma_s\left(\frac{n}{d}\right),
\end{equation}
where $\chi$ is a Dirichlet character modulo $n$ with conductor $d$.

\smallskip

Recently, for any integer $k \geq 1$, T\'{o}th \cite{toth} introduced an arithmetic function $\phi_{k}$ which is a generalization of  the Euler's totient function. The function $\phi_{k}(n)$ is defined by
\begin{equation}
\phi_k(n)= \sum_{\substack{a_1, a_2, \ldots, a_k =1 \\ \gcd(a_1\ldots a_k,n)=1 \\ \gcd(a_1+\cdots+a_k,n)=1}}^n 1.
\end{equation} 
He proved that for every integer $k \geq 1$, the function $\phi_k$ is multiplicative and the Menon-type identity
\begin{equation}\label{toth}
\sum_{\substack{a_1, a_2, \ldots, a_k =1 \\ \gcd(a_1\ldots a_k,n)=1 \\ \gcd(a_1+\cdots+a_k,n)=1}}^n \gcd(a_1+\cdots+a_k-1,n) = \phi_k(n)\sigma_0(n)
\end{equation} holds.

%


Here we note that the function $\phi_2(n)$ was introduced by Arai and Gakuen \cite{ag}, and Carlitz \cite{c} proved the corresponding formula for $\phi_2(n)$. Also for $k=2$, the identity \eqref{toth},
\begin{equation}\label{sr}
\sum_{\substack{a_1, a_2 =1 \\ \gcd(a_1a_2,n)=1 \\ \gcd(a_1+a_2,n)=1}}^n \gcd(a_1+a_2-1,n) = \phi_2(n)\sigma_0(n)
\end{equation}
was deduced by Sita Ramaiah \cite{sr}.

Recently, Ji and Wang \cite{wj} generalized Sita Ramaiah's identity involving a  Dirichlet character. More precisely, they proved that
\begin{equation}\label{wj}
\sum_{\substack{a_1, a_2 =1 \\ \gcd(a_1a_2,n)=1 \\ \gcd(a_1+a_2,n)=1}}^n \gcd(a_1+a_2-1,n)\chi(a_1) =\mu(d) \phi\left(\frac{n_0^2}{d}\right) \phi_2\left(\frac{n}{n_0}\right)\sigma_0\left(\frac{n}{d}\right),
\end{equation}
where $\mu$ is the M\"obius function, $\chi$ is a Dirichlet character modulo $n$ with conductor $d$, the integer $n_{0}$ is such that $n_0 \mid n$ and $n_0$ has the same prime factors as that of $d$ and $\gcd\left(n_0, \frac{n}{n_0}\right) = 1$.

In this article, we generalize the identity \eqref{wj} in the light of Sury's identity \eqref{sury} to an algebraic number field $K$ with ring of integers $\mathcal{O}_{K}$. Before we state the main theorem of our paper, we introduce the fundamental functions defined on the integral ideals of $\mathcal{O}_{K}$. In what follows, for an element $\alpha \in \mathcal{O}_{K}$, the principal ideal $\alpha\mathcal{O}_{K}$ is denoted by $\langle \alpha \rangle$ and $\gcd(\alpha,\mathfrak{a})$ will denote $\gcd(\langle \alpha \rangle,\mathfrak{a})$.
\begin{definition}
Let $K$ be an algebraic number field with ring of integers $\mathcal{O}_{K}$ and let $\mathfrak{a}$ be a non-zero ideal in $\mathcal{O}_{K}$. Then we define the following functions on the set of integral ideals of $K$.
\begin{enumerate}
\item The M\"obius $\mu$ function is defined as 
\begin{equation*}
\mu(\mathfrak{a})=
\begin{cases}
    1   & ~ \text{ if } \mathfrak{a} = \mathcal{O}_{K},\\
    (-1)^{t} &  ~ \text{ if } \mathfrak{a} = \displaystyle\prod_{\i = 1}^{t}\mathfrak{p}_{i} \mbox{ where each } \mathfrak{p}_{i} \mbox{ is a prime ideal and } \mathfrak{p}_{i} \neq \mathfrak{p}_{j} \mbox{ for } i \neq j,\\
    0 & ~ \text{ otherwise. }
\end{cases}
\end{equation*}

\medskip

\item The Euler totient function is defined as $$\phi(\mathfrak{a}) := |(\mathcal{O}_{K}/\mathfrak{a})^{*}| =  N(\mathfrak{a}) \prod_{\mathfrak{p}\mid \mathfrak{a}}\left(1-\frac{1}{N(\mathfrak{p})}\right),$$ where for any ideal $\mathfrak{b}$ in $\mathcal{O}_K$, $N(\mathfrak{b})$ stands for the absolute norm $|\mathcal{O}_{K}/\mathfrak{b}|$ of the ideal $\mathfrak{b}$.

\medskip

\item For an integer $s \geq 0$, the function $\sigma_{s}$ is defined by $$\sigma_{s}(\mathfrak{a}) = \displaystyle\sum_{\mathfrak{d} \mid \mathfrak{a}} N(\mathfrak{d})^{s}.$$ 

\item The function $\phi_{2}(\mathfrak{a})$ is defined as $$\phi_{2}(\mathfrak{a}) = \sum_{\substack{a,b \in (\mathcal{O}_{K}/\mathfrak{a})^{*} \\ a+b \in (\mathcal{O}_{K}/\mathfrak{a})^{*}}} 1.$$
\end{enumerate}
\end{definition}

\begin{remark}
It is also known that (cf. \cite{wj}), $\phi_{2}(\mathfrak{a}) = \phi(\mathfrak{a})^2 \displaystyle\sum_{\mathfrak{d}\mid \mathfrak{a}}\frac{\mu(\mathfrak{d})}{\phi(\mathfrak{d})}.$ 
\end{remark}  
 
Now, we state the main theorem as follows.

\begin{theorem}\label{MAIN-TH}
Let $K$ be an algebraic number field with ring of integers $\mathcal{O}_{K}$ and let $\mathfrak{a}$ be a non-zero ideal in $\mathcal{O}_{K}$. Let $\chi$ be a Dirichlet character modulo $\mathfrak{a}$ with conductor $\mathfrak{d}$ and let $s \geq 0$ be an integer. Then we have 
\begin{equation}
\sum_{\substack{ a_1, a_2, b_1,b_2,\ldots, b_s \in \mathcal{O}_{K}/\mathfrak{a} \\ \gcd(a_1a_2,\mathfrak{a})=1 \\ \gcd(a_1+a_2,\mathfrak{a})=1}} N(\gcd(a_1+a_2-1,b_1,b_2,\ldots, b_s,\mathfrak{a}))\chi(a_1) =\mu(\mathfrak{d}) \phi\left(\frac{\mathfrak{a}_{0}^{2}}{\mathfrak{d}}\right) \phi_2\left(\frac{\mathfrak{a}}{\mathfrak{a}_{0}}\right)\sigma_s\left(\frac{\mathfrak{a}}{\mathfrak{d}}\right).
\end{equation}
where $\mathfrak{a}_{0} \mid \mathfrak{a}$ is such that $\mathfrak{a}_0$ has the same prime ideal factors as $\mathfrak{d}$ and $\left(\mathfrak{a}_0, \frac{\mathfrak{a}}{\mathfrak{a}_0}\right) = 1$.
\end{theorem}

Note that, for $K = \mathbb{Q}$, we obtain the following corollary.

\begin{corollary}\label{mt}
Let $n$ be a positive integer and $\chi$ a Dirichlet character modulo $n$ with conductor $d$. For a non-negative integer $s$, we have the following identity:
\begin{equation}\label{meq}
\sum_{\substack{1 \leq a_1, a_2, b_1,b_2,\ldots, b_s \leq n \\ \gcd(a_1a_2,n)=1 \\ \gcd(a_1+a_2,n)=1}} \gcd(a_1+a_2-1,b_1,b_2,\ldots, b_s,n)\chi(a_1) =\mu(d) \phi\left(\frac{n_0^2}{d}\right) \phi_2\left(\frac{n}{n_0}\right)\sigma_s\left(\frac{n}{d}\right).
\end{equation}
where $\mu(n)$ is the M\"obius function, $n_0 \mid n$ such that $n_0$ has the same prime factors as $d$ and $\left(n_0, \frac{n}{n_0}\right) = 1$.
\end{corollary}

\begin{remark}
If $s=0$, then the identity \eqref{meq} reduces to Ji and Wang's identity \eqref{wj}. Furthermore, if $\chi$ is the trivial character $\chi_{0}$, then \eqref{meq} reduces to
\begin{equation}\label{meq2}
\sum_{\substack{1 \leq a_1, a_2, b_1,b_2,\ldots, b_s \leq n \\ \gcd(a_1a_2,n)=1 \\ \gcd(a_1+a_2,n)=1}} \gcd(a_1+a_2-1,b_1,b_2,\ldots, b_s,n) =  \phi_2(n)\sigma_s(n),
\end{equation}
which is a generalization of Sita Ramaiah's identity \eqref{sr}.
\end{remark}



\section{Preliminaries}
In this section, we fix an algebraic number field $K$ with ring of integers $\mathcal{O}_{K}$ and a non-zero ideal $\mathfrak{a}$ in $\mathcal{O}_{K}$. We first assume that $\mathfrak{a}=\mathfrak{p}^m$, where $\mathfrak{p}$ is a prime ideal and $m \geq 1$ is an integer. Let $\chi$ be a Dirichlet character modulo $\mathfrak{a}$ with conductor $\mathfrak{d}=\mathfrak{p}^t$, where $t$ is an integer with $0 \leq t \leq m$.

Since $\mathfrak{a}=\mathfrak{p}^m$ is a prime power, we note that the additive group $\mathcal{O}_{K}/\mathfrak{a}$ has the following chain of subgroups.
$$0=\mathfrak{p}^m\mathcal{O}_{K}/\mathfrak{a} \subsetneq \mathfrak{p}^{m-1}\mathcal{O}_{K}/\mathfrak{a} \subsetneq \ldots \subsetneq \mathfrak{p}\mathcal{O}_{K}/\mathfrak{a} \subsetneq \mathcal{O}_{K}/\mathfrak{a}.$$

Similarly the multiplicative group $(\mathcal{O}_{K}/\mathfrak{a})^*$ has a filtration consisting of multiplicative subgroups: 
$$1=1+\mathfrak{p}^m\mathcal{O}_{K}/\mathfrak{a} \subsetneq 1+\mathfrak{p}^{m-1}\mathcal{O}_{K}/\mathfrak{a} \subsetneq \ldots \subsetneq 1+\mathfrak{p}\mathcal{O}_{K}/\mathfrak{a} \subsetneq (\mathcal{O}_{K}/\mathfrak{a})^*,$$ with $|1+\mathfrak{p}^k\mathcal{O}_{K}/\mathfrak{a}|=N(\mathfrak{p})^{m-k}$ for $1\leq k \leq m$.  


Now, we consider 
\begin{eqnarray}\label{eqn}
\mathcal{S}_{\chi}(\mathfrak{p}^m,s) & = &\sum_{\substack{ a_1, a_2, b_1,b_2,\ldots, b_s \in \mathcal{O}_{K}/\mathfrak{a} \\ \gcd(a_1a_2,\mathfrak{p}^m)=1 \\ \gcd(a_1+a_2,\mathfrak{p}^m)=1}} N(\gcd(a_1+a_2-1,b_1,b_2,\ldots, b_s,\mathfrak{p}^m))\chi(a_1)\nonumber \\
&=& \sum_{r=0}^m \sum_{\substack{ b_1,b_2,\ldots, b_s \in \mathcal{O}_{K}/\mathfrak{a} \\ \gcd(b_1,b_2,\ldots, b_s,\mathfrak{p}^m) = \mathfrak{p}^r}}\sum_{\substack{ a_1,a_2 \in \mathcal{O}_{K}/\mathfrak{a} \\ \gcd(a_1a_2,\mathfrak{p}^m)=1 \\ \gcd(a_1+a_2,\mathfrak{p}^m)=1}} N(\gcd(a_1+a_2-1, \mathfrak{p}^r)) \chi(a_1) \nonumber \\
&=& \sum_{r=0}^m \left(\sum_{\substack{ a_1,a_2 \in \mathcal{O}_{K}/\mathfrak{a}  \\ \gcd(a_1a_2,\mathfrak{p}^m)=1 \\ \gcd(a_1+a_2,\mathfrak{p}^m)=1}} N(\gcd(a_1+a_2-1, \mathfrak{p}^r)) \chi(a_1)\right) \left(\sum_{\substack{ b_1,b_2,\ldots, b_s \in \mathcal{O}_{K}/\mathfrak{a} \\ \gcd(b_1,b_2,\ldots, b_s,\mathfrak{p}^m) = \mathfrak{p}^r}} 1\right)
\end{eqnarray}

Thus to prove Theorem \ref{MAIN-TH}, we need to compute the sums 

$$\sum_{\substack{ a_1,a_2 \in \mathcal{O}_{K}/\mathfrak{a} \\ \gcd(a_1a_2,\mathfrak{p}^m)=1 \\ \gcd(a_1+a_2,\mathfrak{p}^m)=1}} N(\gcd(a_1+a_2-1, \mathfrak{p}^r)) \chi(a_1) \; \text{ and } \; \sum_{\substack{ b_1,b_2,\ldots, b_s \in \mathcal{O}_{K}/\mathfrak{a} \\ \gcd(b_1,b_2,\ldots, b_s,\mathfrak{p}^m) = \mathfrak{p}^r}} 1.$$ 


Next, we record some lemmas that will be used to compute the sum $\mathcal{S}_{\chi}(\mathfrak{p}^{m},s)$.

\begin{lemma}\cite{wj}\label{fl}
Let $\mathfrak{a}=\mathfrak{p}^m$ and let $\chi$ be a Dirichlet character modulo $\mathfrak{a}$ with conductor $\mathfrak{d}=\mathfrak{p}^t$ with $0 \leq t \leq m$. Then for every integer $1 \leq k \leq m$, we have
\begin{equation*}
\sum_{a \in 1+\mathfrak{p}^k (\mathcal{O}_{K}/\mathfrak{a})^*} \chi(a)=
\begin{cases}
N(\mathfrak{p})^{m-k} & \text{ for } k=t,t+1,\ldots, m \\
0 & \text{ otherwise }.
\end{cases}
\end{equation*}
\end{lemma}
%
%

\begin{lemma} \cite{wj} \label{f2}
Let $\mathfrak{a}=\mathfrak{p}^m$ and $\chi$ a Dirichlet character modulo $\mathfrak{a}$ with conductor $\mathfrak{d}=\mathfrak{p}^t$, for some $0 \leq t \leq m$. Then for an integer $k$ with $1\leq k \leq m$, we have
\begin{equation*}
\sum_{\substack{a_{1}, a_{2} \in (\mathcal{O}_{K}/\mathfrak{a})^{*} \\ a_1+a_2\equiv 1 \pmod {\mathfrak{p}^k}}} \chi(a_1)=
\begin{cases}
    N(\mathfrak{p})^{m-k}(N(\mathfrak{p})^m-2N(\mathfrak{p})^{m-1})   & ~ \text{ if }t=0\\
    -N(\mathfrak{p})^{2m-k-1} &  ~ \text{ if }t=1\\
    0 & ~ \text{ otherwise. }
\end{cases}
\end{equation*}
\end{lemma}

\begin{lemma} \cite{wj}\label{f3}
Let $\mathfrak{a}=\mathfrak{p}^m$ and $\chi$ a Dirichlet character modulo $\mathfrak{a}$ with conductor $\mathfrak{d}=\mathfrak{p}^t$, for some $0 \leq t \leq m$. Then we have
\begin{equation*}
\sum_{\substack{a_{1}, a_{2} \in (\mathcal{O}_{K}/\mathfrak{a})^{*} \\ \gcd(a_1+a_2,\mathfrak{p}^m) = 1\\ \gcd(a_1+a_2-1,\mathfrak{p}^m)=1}} \chi(a_1)=
\begin{cases}
    (N(\mathfrak{p})^m-N(\mathfrak{p})^{m-1}) (N(\mathfrak{p})^m-3N(\mathfrak{p})^{m-1})+N(\mathfrak{p})^{2m-2}  & ~ \text{ if }t=0\\
    N(\mathfrak{p})^{2m-2} &  ~ \text{ if }t=1\\
    0 & ~ \text{ otherwise. }
\end{cases}
\end{equation*}
\end{lemma}

The next two lemmas are of utmost importance in the proof of Theorem \ref{MAIN-TH} and we give detailed proofs here.

\begin{lemma} \label{lfs1}
Let $\mathfrak{a}=\mathfrak{p}^m$ and $\chi$ a Dirichlet character modulo $\mathfrak{a}$ with conductor $\mathfrak{d}=\mathfrak{p}^t$, where $0 \leq t \leq m$. Let $r$ be an integer such that $0 \leq r \leq m$. Then
\begin{equation}\label{elfs1}
\sum_{\substack{ a_1,a_2 \in (\mathcal{O}_{K}/\mathfrak{a})^{*} \\ \gcd(a_1+a_2,\mathfrak{p}^m)=1}} N(\gcd(a_1+a_2-1, \mathfrak{p}^r))\chi(a_1)=
\begin{cases}
    (r+1)\phi_2(\mathfrak{p}^m)   & ~ \text{ if }t=0\\
    -r\phi(\mathfrak{p}^{2m-1}) & ~  \text{ if }t=1\\
    0 & ~ \text{ otherwise. }
\end{cases}
\end{equation}
\end{lemma}

\begin{proof}
We denote the sum in the LHS of \eqref{elfs1} by $\mathcal{S}$. Then
\begin{eqnarray*}
\mathcal{S} &=& \sum_{i=0}^m N(\mathfrak{p}^i) \sum_{\substack{a_1,a_2 \in (\mathcal{O}_{K}/\mathfrak{a})^{*} \\ \gcd(a_1+a_2,\mathfrak{p}^m)=1 \\\gcd(a_1+a_2-1,\mathfrak{p}^r)=\mathfrak{p}^i}} \chi(a_1) \\
&=&\sum_{\substack{a_1,a_2 \in (\mathcal{O}_{K}/\mathfrak{a})^{*} \\ \gcd(a_1+a_2,\mathfrak{p}^m)=1\\\gcd(a_1+a_2-1,\mathfrak{p}^m)=1}} \chi(a_1) \;+\; \sum_{i=1}^{r-1} N(\mathfrak{p})^i \sum_{\substack{a_1,a_2 \in (\mathcal{O}_{K}/\mathfrak{a})^{*} \\ \gcd(a_1+a_2,\mathfrak{p}^m)=1 \\ \gcd(a_1+a_2-1,\mathfrak{p}^r)=\mathfrak{p}^i}} \chi(a_1) \;+\; N(\mathfrak{p})^r \sum_{i=r}^m\sum_{\substack{a_1,a_2 \in (\mathcal{O}_{K}/\mathfrak{a})^{*} \\ \gcd(a_1+a_2,\mathfrak{p}^m)=1 \\ \gcd(a_1+a_2-1,\mathfrak{p}^r)=\mathfrak{p}^r}} \chi(a_1) \\
&=& \sum_{\substack{a_1,a_2 \in (\mathcal{O}_{K}/\mathfrak{a})^{*} \\ \gcd(a_1+a_2,\mathfrak{p}^m)=1\\ \gcd(a_1+a_2-1,\mathfrak{p}^m)=1}} \chi(a_1) \; +\; \sum_{i=1}^{r-1} N(\mathfrak{p})^i \sum_{\substack{a_1,a_2 \in (\mathcal{O}_{K}/\mathfrak{a})^{*} \\ \gcd(a_1+a_2,\mathfrak{p}^m)=1 \\ \gcd(a_1+a_2-1,\mathfrak{p}^r)=\mathfrak{p}^i}} \chi(a_1) \;+\; N(\mathfrak{p})^r \sum_{\substack{a_1,a_2 \in (\mathcal{O}_{K}/\mathfrak{a})^{*} \\ a_1+a_2 \in 1 + \mathfrak{p}^r/\mathfrak{p}^{m}}} \chi(a_1) \\
&=& \sum_{\substack{a_1,a_2 \in (\mathcal{O}_{K}/\mathfrak{a})^{*} \\ \gcd(a_1+a_2,\mathfrak{p}^m)=1\\
\gcd(a_1+a_2-1,\mathfrak{p}^m)=1}} \chi(a_1) \; +\; \sum_{i=1}^{r-1} N(\mathfrak{p})^i \Bigg(\sum_{\substack{ a_1,a_2 \in (\mathcal{O}_{K}/\mathfrak{a})^{*} \\a_1+a_2 \in 1+\mathfrak{p}^i/\mathfrak{p}^{m}}} \chi(a_1) 
- \sum_{\substack{ a_1,a_2 \in (\mathcal{O}_{K}/\mathfrak{a})^{*} \\a_1+a_2 \in 1+\mathfrak{p}^{i+1}/\mathfrak{p}^{m}}} \chi(a_1)\Bigg) \\ &+& N(\mathfrak{p})^r \sum_{\substack{ a_1,a_2 \in (\mathcal{O}_{K}/\mathfrak{a})^{*} \\a_1+a_2 \in 1 + \mathfrak{p}^r/\mathfrak{p}^{m}}} \chi(a_1) \\
&=& \sum_{\substack{a_1,a_2 \in (\mathcal{O}_{K}/\mathfrak{a})^{*} \\ \gcd(a_1+a_2,\mathfrak{p}^m)=1\\ \gcd(a_1+a_2-1,\mathfrak{p}^m)=1}} \chi(a_1) \; +\; \sum_{i=1}^{r} N(\mathfrak{p})^i \sum_{\substack{ a_1,a_2 \in (\mathcal{O}_{K}/\mathfrak{a})^{*} \\a_1+a_2 \in 1+\mathfrak{p}^i/\mathfrak{p}^{m}}} \chi(a_1) - \sum_{i=1}^{r-1} N(\mathfrak{p})^i \sum_{\substack{ a_1,a_2 \in (\mathcal{O}_{K}/\mathfrak{a})^{*} \\a_1+a_2 \in 1+\mathfrak{p}^{i+1}/\mathfrak{p}^{m}}} \chi(a_1)
\end{eqnarray*}


{\bf Case 1. $t=0$}. By using Lemma \ref{f2} and Lemma \ref{f3}, we have
\begin{eqnarray*}
\mathcal{S} &=& (N(\mathfrak{p})^m-N(\mathfrak{p})^{m-1})(N(\mathfrak{p})^m-3N(\mathfrak{p})^{m-1})+N(\mathfrak{p})^{2m-2} \\ &+& \sum_{i=1}^r N(\mathfrak{p})^i N(\mathfrak{p})^{m-i}(N(\mathfrak{p})^m-2N(\mathfrak{p})^{m-1}) - \sum_{i=1}^{r-1} N(\mathfrak{p})^i N(\mathfrak{p})^{m-i-1}(N(\mathfrak{p})^m-2N(\mathfrak{p})^{m-1}) \\
&=& (N(\mathfrak{p})^m-N(\mathfrak{p})^{m-1})(N(\mathfrak{p})^m-3N(\mathfrak{p})^{m-1})+N(\mathfrak{p})^{2m-2} + r N(\mathfrak{p})^m(N(\mathfrak{p})^m-2N(\mathfrak{p})^{m-1})\\ &-& (r-1)N(\mathfrak{p})^{m-1}(N(\mathfrak{p})^m-2N(\mathfrak{p})^{m-1}) \\
&=& (r+1)(N(\mathfrak{p})^m-N(\mathfrak{p})^{m-1})(N(\mathfrak{p})^m-2N(\mathfrak{p})^{m-1}) \\
&=& (r+1)(N(\mathfrak{p})^m-N(\mathfrak{p})^{m-1})^2 -(r+1)(N(\mathfrak{p})^m-N(\mathfrak{p})^{m-1})N(\mathfrak{p})^{m-1} \\
&=& (r+1) \phi(\mathfrak{p}^m)^2 \left(1-\frac{N(\mathfrak{p})^{m-1}}{N(\mathfrak{p})^m-N(\mathfrak{p})^{m-1}}\right) \\
&=& (r+1)\phi(\mathfrak{p}^m)^2 \left(1-\frac{1}{N(\mathfrak{p})-1}\right) \\
&=& (r+1)\phi_2(\mathfrak{p}^m)
\end{eqnarray*}

{\bf Case 2. $t=1$}. Again, by using Lemma \ref{f2} and Lemma \ref{f3}, we see that
\begin{eqnarray*}
\mathcal{S} &=& N(\mathfrak{p})^{2m-2}-\sum_{i=1}^r N(\mathfrak{p})^i N(\mathfrak{p})^{2m-i-1}+ \sum_{i=1}^{r-1} N(\mathfrak{p})^i N(\mathfrak{p})^{2m-i-2} \\
&=& N(\mathfrak{p})^{2m-2} - r N(\mathfrak{p})^{2m-1}+(r-1)N(\mathfrak{p})^{2m-2} \\
&=& -r(N(\mathfrak{p})^{2m-1}-N(\mathfrak{p})^{2m-2}) \\
&=& -r \phi(\mathfrak{p}^{2m-1})
\end{eqnarray*}

{\bf Case 3. $t\geq 2$}. In this case also, by using Lemma \ref{f2} and Lemma \ref{f3}, we get $\mathcal{S}=0$.

This completes the proof of the lemma.
\end{proof}

\begin{lemma}\label{lfs2}
Let $\mathfrak{a}=\mathfrak{p}^m$ and let $r$ be an integer with $0 \leq r \leq m$. Then for every integer $s \geq 0$, we have
\begin{equation*}
\displaystyle\sum_{\substack{ b_1,b_2,\ldots, b_s \in \mathcal{O}_{K}/\mathfrak{a} \\ \gcd(b_1,b_2,\ldots, b_s,\mathfrak{p}^m) = \mathfrak{p}^r}} 1 =
\begin{cases}
N(\mathfrak{p})^{(m-r)s}-N(\mathfrak{p})^{(m-r-1)s} & \text{ when } r<m \\
1 & \text{ when } r=m.
\end{cases}
\end{equation*}
\end{lemma}

\begin{proof}
For $0 \leq r \leq m-1$, we note that
$$\gcd(b_1,b_2,\ldots, b_s,\mathfrak{p}^m)=\mathfrak{p}^r \; \text{ if and only if }\; (b_1,b_2,\ldots, b_s) \in (\mathfrak{p}^r/\mathfrak{p}^{m})^s-(\mathfrak{p}^{r+1}/\mathfrak{p}^{m})^s.$$

And for $r=m$,
$$\gcd(b_1,b_2,\ldots, b_s,\mathfrak{p}^m)=\mathfrak{p}^m \; \text{ if and only if }\; (b_1,b_2,\ldots, b_s) \in (\mathfrak{p}^m/\mathfrak{p}^{m})^s.$$

Since $|(\mathfrak{p}^r/\mathfrak{p}^{m})|=N(\mathfrak{p})^{m-r}$ and $\mathfrak{p}^{r}/\mathfrak{p}^{m} \supsetneq \mathfrak{p}^{r + 1}/\mathfrak{p}^{m}$, we get the desired result.
\end{proof}


We now prove the following theorem for a power of a prime ideal.

\begin{theorem}\label{mtp}
Let $\mathfrak{a} = \mathfrak{p}^m$ and $\chi$ be a Dirichlet character modulo $\mathfrak{a}$ with conductor $\mathfrak{d} = \mathfrak{p}^t$, with $0 \leq t \leq m$. Let $s \geq 0$ be an integer. Then we have
\begin{equation}\label{yoman}
\sum_{\substack{ a_1, a_2, b_1,b_2,\ldots, b_s \in \mathcal{O}_{K}/\mathfrak{a} \\ \gcd(a_1a_2,\mathfrak{a})=1 \\ \gcd(a_1+a_2,\mathfrak{a})=1}} N(\gcd(a_1+a_2-1,b_1,b_2,\ldots, b_s,\mathfrak{a}))\chi(a_1) =\mu(\mathfrak{d}) \phi\left(\frac{\mathfrak{a_{0}}^2}{\mathfrak{d}}\right) \phi_2\left(\frac{\mathfrak{a}}{\mathfrak{a}_0}\right)\sigma_s\left(\frac{\mathfrak{a}}{\mathfrak{d}}\right).
\end{equation}
where $\mathfrak{a}_{0}$ is a divisor of $\mathfrak{a}$ such that $\mathfrak{a}_{0}$ has the same set of prime divisors as that of $\mathfrak{d}$ and $\gcd \left(\mathfrak{a}_{0}, \frac{\mathfrak{a}}{\mathfrak{a}_{0}}\right) = 1$ and $\mu$ is the M\"obius function.
\end{theorem}
\begin{proof}
We denote the left hand side of equation \eqref{yoman} by $\mathcal{S}_{\chi}(\mathfrak{p}^m,s)$. Then from the equation \eqref{eqn}, we get 


\begin{equation}\label{JUST ABOVE}
\mathcal{S}_{\chi}(\mathfrak{p}^m,s) = \sum_{r=0}^m \left(\sum_{\substack{ a_1,a_2 \in (\mathcal{O}_{K}/\mathfrak{a})^{*} \\ \gcd(a_1+a_2,\mathfrak{p}^m)=1}} N(\gcd(a_1+a_2-1,\mathfrak{p}^{r})) \chi(a_1)\right) \sum_{\substack{ b_1,b_2,\ldots, b_s \in \mathcal{O}_{K}/\mathfrak{a} \\ \gcd(b_1,b_2,\ldots, b_s,\mathfrak{p}^m) = \mathfrak{p}^r}} 1
\end{equation}


{\bf Case 1. $t=0$}. By \eqref{JUST ABOVE} and Lemma \ref{lfs1} and Lemma \ref{lfs2}, we get

\begin{eqnarray*}
\mathcal{S}_{\chi}(\mathfrak{p}^m,s) &=& 
\sum_{r=0}^m \left(\sum_{\substack{a_1,a_2 (\mathcal{O}_{K}/\mathfrak{a})^{*} \\ \gcd(a_1+a_2,\mathfrak{p}^m)=1}} N(\gcd(a_1+a_2-1, \mathfrak{p}^r))\chi(a_1)\right) \sum_{\substack{ b_1,b_2,\ldots, b_s \mathcal{O}_{K}/\mathfrak{a} \\ \gcd(b_1,b_2,\ldots, b_s,\mathfrak{p}^m) = \mathfrak{p}^r}} 1 \\
&=& \sum_{r=0}^m (r+1)\phi_2(\mathfrak{p}^m)\sum_{\substack{ b_1,b_2,\ldots, b_s \mathcal{O}_{K}/\mathfrak{a} \\ \gcd(b_1,b_2,\ldots, b_s,\mathfrak{p}^m) = \mathfrak{p}^r}} 1 \\
&=& \sum_{r=0}^{m-1}(r+1)\phi_2(\mathfrak{p}^m)\left(N(\mathfrak{p})^{(m-r)s}-N(\mathfrak{p})^{(m-r-1)s}\right)+(m+1)\phi_2(\mathfrak{p}^m) \\
&=& \sum_{r=0}^{m}(r+1)\phi_2(\mathfrak{p}^m)N(\mathfrak{p})^{(m-r)s} - \sum_{r=0}^{m-1}(r+1)\phi_2(\mathfrak{p}^m)N(\mathfrak{p})^{(m-r-1)s} \\
&=& \phi_2(\mathfrak{p}^m) \left[\sum_{r=0}^{m}(r+1)N(\mathfrak{p})^{(m-r)s} - \sum_{r=1}^{m}rN(\mathfrak{p})^{(m-r)s}\right] \\
&=& \phi_2(\mathfrak{p}^m)\sum_{r=0}^m N(\mathfrak{p})^{(m-r)s} \\
&=& \phi_2(\mathfrak{p}^m)\sum_{r=0}^m N(\mathfrak{p})^{rs}\\
&=& \phi_2(\mathfrak{p}^m) \sigma_s(\mathfrak{p}^m).
\end{eqnarray*}

{\bf Case 2. $t=1$.} By \eqref{JUST ABOVE} and Lemma \ref{lfs1} and Lemma \ref{lfs2}, we get 
\begin{eqnarray*}
\mathcal{S}_{\chi}(\mathfrak{p}^m,s) &=& \sum_{r=0}^m \left(\sum_{\substack{a_1,a_2 \in (\mathcal{O}_{K}/\mathfrak{a})^{*} \\ \gcd(a_1+a_2,\mathfrak{p}^m)=1}} N(\gcd(a_1+a_2-1, \mathfrak{p}^r)) \chi(a_1)\right) \sum_{\substack{ b_1,b_2,\ldots, b_s \in \mathcal{O}_{K}/\mathfrak{a} \\ \gcd(b_1,b_2,\ldots, b_s,\mathfrak{p}^m) = \mathfrak{p}^r}} 1 \\
&=& \sum_{r=0}^m -r\phi(\mathfrak{p}^{2m-1})\sum_{\substack{b_1,b_2,\ldots, b_s \in \mathcal{O}_{K}/\mathfrak{a} \\ \gcd(b_1,b_2,\ldots, b_s,\mathfrak{p}^m) = \mathfrak{p}^r}} 1 \\
&=&- \sum_{r=1}^{m-1}r\phi(\mathfrak{p}^{2m-1})\left(N(\mathfrak{p})^{(m-r)s}-N(\mathfrak{p})^{(m-r-1)s}\right)-m\phi(\mathfrak{p}^{2m-1}) \\
&=& -\phi(\mathfrak{p}^{2m-1}) \left[\sum_{r=1}^m rN(\mathfrak{p)}^{(m-r)s} - \sum_{r=1}^{m-1} rN(\mathfrak{p})^{(m-r-1)s}\right] \\
&=& -\phi(\mathfrak{p}^{2m-1}) \left[\sum_{r=1}^m rN(\mathfrak{p})^{(m-r)s} - \sum_{r=2}^{m} (r-1)N(\mathfrak{p)}^{(m-r)s}\right] \\
&=& -\phi(\mathfrak{p}^{2m-1}) \left[N(\mathfrak{p})^{(m-1)s} + \sum_{r=2}^m N(\mathfrak{p})^{(m-r)s}\right] \\
&=& -\phi(\mathfrak{p}^{2m-1})\sum_{r=1}^m N(\mathfrak{p})^{(m-r)s} \\
&=& -\phi(\mathfrak{p}^{2m-1})\sigma_s(\mathfrak{p}^{m-1}).
\end{eqnarray*}

{\bf Case 3. $t\geq 2$.} By \eqref{JUST ABOVE} and Lemma \ref{lfs1}, we get 
\begin{eqnarray*}
\mathcal{S}_{\chi}(\mathfrak{p}^m,s) &=& \sum_{r=0}^m \left(\sum_{\substack{ a_1,a_2 \in (\mathcal{O}_{K}/\mathfrak{a})^{*} \\ \gcd(a_1+a_2,\mathfrak{p}^m)=1}} N(\gcd(a_1+a_2-1, \mathfrak{p}^r)) \chi(a_1)\right) \sum_{\substack{ b_1,b_2,\ldots, b_s \in \mathcal{O}_{K}/\mathfrak{a} \\ \gcd(b_1,b_2,\ldots, b_s,\mathfrak{p}^m) = \mathfrak{p}^r}} 1 \\
&=& 0.
\end{eqnarray*}

\end{proof}

\section{Proof of Theorem \ref{MAIN-TH}}
We first prove that $\mathcal{S}_{\chi}(\mathfrak{a},s)$ is a multiplicative function in the first variable. Then, using the
multiplicative property, we prove Theorem \ref{mt} by combining with the prime power case.

Let $\mathfrak{a}=\mathfrak{a}_{1}\mathfrak{a_{2}}$, where $\mathfrak{a}_{1}$ and $\mathfrak{a}_{2}$ are two integral ideals of $\mathcal{O}_{K}$ such that $\gcd(\mathfrak{a}_{1},\mathfrak{a}_{2}) = 1$. Then by the Chinese Remainder Theorem, we have the following isomorphism
$$\Psi : (\mathcal{O}_{K}/\mathfrak{a})^{*} \rightarrow (\mathcal{O}_{K}/\mathfrak{a}_{1})^{*} \times (\mathcal{O}_{K}/\mathfrak{a}_{2})^{*}$$ given by  \hspace*{3cm}
$\Psi(a)=(a^{'},a^{''})$, \; where $a \equiv a{'} \pmod {\mathfrak{a}_{1}}$ and $a \equiv a{''} \pmod {\mathfrak{a}_{2}}$.

\bigskip

Let
$$\mathcal{S}(\mathfrak{a}) = \{ (a_1,a_2): a_1,a_2, a_1+a_2 \in (\mathcal{O}_{K}/\mathfrak{a})^{*}\},$$ $$\mathcal{S}(\mathfrak{a}_{1}) = \{ (a_1,a_2): a_1,a_2, a_1+a_2 \in (\mathcal{O}_{K}/\mathfrak{a}_{1})^{*}\}$$ and $$\mathcal{S}(\mathfrak{a}_{2}) = \{ (a_1,a_2): a_1,a_2, a_1+a_2 \in (\mathcal{O}_{K}/\mathfrak{a}_{2})^{*}\},$$ 

We define the map $f : \mathcal{S}(\mathfrak{a}) \rightarrow \mathcal{S}(\mathfrak{a}_{1})\times \mathcal{S}(\mathfrak{a}_{2})$, given by $f((a_1,a_2)) = (\Psi(a_1),\Psi(a_2))$. Since $\Psi$ is an isomorphism of groups, it follows that $f$ is a bijection. Also, every Dirichlet character $\chi$ modulo $\mathfrak{a}$ can be uniquely factored as a product of the form $\chi= \chi_1 \chi_2$, where $\chi_1$ and $\chi_2$ are Dirichlet characters modulo $\mathfrak{a}_{1}$ and $\mathfrak{a}_{2}$, respectively. Therefore, we have
\begin{eqnarray*}
\mathcal{S}_{\chi}(\mathfrak{a},s) &=& \sum_{\substack{ a_1, a_2, b_1,b_2,\ldots, b_s \in \mathcal{O}_{K}/\mathfrak{a} \\ a_{1}a_{2}, a_{1} + a_{2} \in (\mathcal{O}_{K}/\mathfrak{a})^{*}}} N(\gcd(a_1+a_2-1,b_1,b_2,\ldots, b_s,\mathfrak{a}))\chi(a_1) \\
&=& \sum_{\substack{ a_1, a_2, b_1,b_2,\ldots, b_s \in \mathcal{O}_{K}/\mathfrak{a}_{1} \\ a_{1}a_{2}, a_{1} + a_{2} \in (\mathcal{O}_{K}/\mathfrak{a}_{1})^{*}}} N(\gcd(a_1+a_2-1,b_1,b_2,\ldots, b_s,\mathfrak{a}_{1}))\chi_{1}(a_1) \\
& \times & \sum_{\substack{ a_1, a_2, b_1,b_2,\ldots, b_s \in \mathcal{O}_{K}/\mathfrak{a}_{2} \\ a_{1}a_{2}, a_{1} + a_{2} \in (\mathcal{O}_{K}/\mathfrak{a}_{2})^{*}}} N((a_1+a_2-1,b_1,b_2,\ldots, b_s,\mathfrak{a}_{2}))\chi_{2}(\mathfrak{a}_{2}) \\
&=& \mathcal{S}_{\chi_{1}}(\mathfrak{a}_{1},s)\mathcal{S}_{\chi_{2}}(\mathfrak{a}_{2},s).
\end{eqnarray*}
Now, we complete the proof of Theorem \ref{MAIN-TH}. Let $\mathfrak{a} = \mathfrak{p}_{1}^{e_1}\mathfrak{p}_{2}^{e_2}\ldots \mathfrak{p}_{t}^{e_t}$ be the factorization of $\mathfrak{a}$ into product of prime ideals and let $\chi$ be a Dirichlet character modulo $\mathfrak{a}$ with conductor $\mathfrak{d}$. Then $\chi$ can be uniquely written as $\chi = \chi_1 \chi_2 \ldots \chi_t$, 
where for each $j \in \{1,\ldots ,t\}$, the character $\chi_j$ is a Dirichlet character modulo $\mathfrak{p}_{j}^{e_j}$ with conductor $\mathfrak{d}_{j}$, for some $\mathfrak{d}_{j} \mid \mathfrak{p}_{j}^{e_{j}}$. Then we have $\mathfrak{d} = \mathfrak{d}_{1} \mathfrak{d}_{2}\ldots \mathfrak{d}_{t}$. Consequently, by using Theorem \ref{mtp}, we obtain
\begin{eqnarray*}
\mathcal{S}_{\chi}(\mathfrak{a},s) &=& \prod_{i=1}^t \mathcal{S}_{\chi_i}(\mathfrak{p}_{i}^{e_i},s)\\
&=& \mu(\mathfrak{d}) \phi\left(\frac{\mathfrak{a}_{0}^2}{\mathfrak{d}}\right) \phi_2\left(\frac{\mathfrak{a}}{\mathfrak{a}_0}\right)\sigma_s\left(\frac{\mathfrak{a}}{\mathfrak{d}}\right).
\end{eqnarray*}
where $\mathfrak{a}_{0}$ is a divisor of $\mathfrak{a}$ such that $\mathfrak{a}_{0}$ has the same set of prime divisors as that of $\mathfrak{d}$ and $\gcd \left(\mathfrak{a}_{0}, \frac{\mathfrak{a}}{\mathfrak{a}_{0}}\right) = 1$. This completes the proof of Theorem \ref{MAIN-TH}. $\hfill\Box$

\medskip

{\bf Acknowledgements.} It is a pleasure to thank Prof. R. Thangadurai for carefully reading the manuscript and giving us some valuable suggestions that improved the readability of the paper. The first author thanks Indian Institute of Technology, Guwahati and the second author thanks Ramakrishna Mission Vivekananda Educational and Research Institute, Belur Math for providing financial support.

%
%

\end{document}